\documentclass[12pt]{article}
\usepackage{amsmath,amssymb,amsthm,euscript,amsfonts,amscd}

\newtheorem{definition}{Definition}
\newtheorem{theorem}{Theorem}
\newtheorem{lemma}{Lemma}
\newtheorem{prop}{Proposition}
\newtheorem{corollary}{Corollary}

\newcommand{\p}[2]{\langle {#1},\ {#2} \rangle}
\newcommand{\m}[2]{\ensuremath{({#1}\times {#2})}}
\newcommand{\s}[2]{\ensuremath{\mathbb R^{#1}_{#2}}}
\newcommand{\es}{\EuScript }
\newcommand{\x}{\ensuremath{\EuScript{X}}}
\newcommand{\y}{\ensuremath{\EuScript{Y}}}
\newcommand{\pr}[2]{\ensuremath{\mathrm{proj}_{\mathcal{#1}} \EuScript {#2} }}
\newcommand{\prj}[2]{\ensuremath{\mathrm{proj}_{#1} \EuScript {#2} }}
\newcommand{\e}{\ensuremath{\mathrm{E}\,}}
\newcommand{\var}{\ensuremath{\mathop{\mathrm{Var}}}}
\newcommand{\cov}[2]{\ensuremath{\mathrm{Cov}({#1},{#2})}}

\numberwithin{equation}{section}

\newcommand{\BT}{\mathbf{T}}
\newcommand{\BR}{\mathbf{R}}

\newcommand{\Mr}{\mathbb R}

\newcommand{\ML}{\mathcal L}

\newcommand{\al}{\alpha}
\newcommand{\bs}{\bfseries}
\newcommand{\be}{\beta}
\newcommand{\ga}{\gamma}

\newcommand{\la}{\lambda}

\newcommand{\OL}{\overline}

\newcommand{\SL}{\sum\limits}

\DeclareMathOperator{\tr}{tr}

 \DeclareMathOperator{\proj}{proj}

\begin{document}

\author{Tyurin, Y. N.}
\title{\large{\bf MULTIVARIATE STATISTICAL ANALYSIS: A GEOMETRIC
PERSPECTIVE}}
\maketitle

\tableofcontents

\newpage
\section*{Introduction}
\addcontentsline{toc}{section}{Introduction}

Linear statistical analysis, and the least squares method
specifically, achieved their modern complete form in the language of
linear algebra, that is in the language of geometry.  In this article
we will show that {\it multivariate} linear statistical analysis in
the language of geometry can be stated just as beautifully and
clearly.  In order to do this, the standard methods of linear algebra
must be expanded.  The first part of this article introduces this
generalization of linear algebra.  The second part introduces the
theory of multivariate statistical analysis in the first part's
language.  We believe that until now multivariate statistical
analysis, though explained in dozens of textbooks, has not had
adequate forms of expression.

Multivariate observations are the observations of several random
quantities in one random experiment.  We shall further record
multivariate observations as columns.  We commonly provide
multivariate observations with indices.  In the simple case, natural
numbers serve as the indices.  (This can be the numbers of
observations in the order they were recorded).  For independent evenly
distributed observations this is a fitting way to organize
information.  If the distributions of observation depend on one or
more factors, the values or combinations of values of these factors
can serve as indices.  Commonly the levels of factors are numbered.
In that case the index is the set of numbers.  So, in a two-factor
scheme (classification by two traits) pairs of natural numbers serve
as indices.

We shall call the set of observations, provided with indices and so
organized, an {\it array}.

For theoretical analysis the linear numeration of data is most
convenient.  Further we will be holding to this system.  When
analyzing examples we will return, if needed, to the natural indexing
of data.

In univariate statistical analysis the numeration of data allows
recording as rows.  In the multivariate case the entirety of the
enumerated data (that is arrays) can also be examined as a row of
columns.  In many cases (but not always) such an array can be treated
as a matrix.

Arrays of one form naturally form a vector space under the operation
of addition and multiplication by numbers.  For the purposes of
statistical analysis this vector space is given a scalar product.  In
one dimensional analysis, if the observations are independent and have
equivalent dispersions, then the most fitting scalar product is the
euclidean product.  In more detail: let the observations have an index
$\alpha$; let arrays $\mathbf{T}_X$ and $\mathbf{T}_Y$ be composed of
the one-dimensional elements $X_\alpha$, $Y_\alpha$.  Then the
euclidean scalar product of arrays $\mathbf{T}_X$ and $\mathbf{T}_Y$
is
\begin{equation}
    \p{\mathbf{T}_X}{\mathbf{T}_Y} = \sum_\alpha X_\alpha Y_\alpha,
\end{equation}
where the index of summation goes through all possible values.
We shall record multivariate observations as columns.
In the multivariate case, the elements $X_\alpha, Y_\alpha$ are columns.
For arrays composed of columns, let us accept the following definition
of the scalar product of arrays $\mathbf{T}_X$ and $\mathbf{T}_Y$:
\begin{equation}\label{ast}
    \p{\mathbf{T}_X}{\mathbf{T}_Y} = \sum_\alpha X_\alpha Y_\alpha^T.
\end{equation}

The scalar product (\ref{ast}) is a square matrix.
Therefore, for arrays composed of columns, square matrices of the
corresponding dimensions must play the role of scalars.
With the help of the scalar product (\ref{ast}) and its consequences,
this article develops a theory of multivariate statistical
analysis, parallel to existing well-known univariate theory.

%----------------------------------------------------------------------------
\section{Modules of Arrays \\ Over a Ring of Matrices}

\subsection{Space of Arrays}

In the introduction we agreed to hold to a linear order of indexation
for simplicity's sake.  However, all the introduced theorems need only
trivial changes to apply to arrays with a different indexation.

Let us consider a $p$-dimensional array with $n$ elements,
\begin{equation}\label{(1.1)}
\BT := \{X_i\mid i = \OL{1, n}\},
\end{equation}
where $X_1, X_2, \ldots, X_n$ are $p$-dimensional vector-columns.
Arrays of this nature form a linear space with addition and
multiplication by numbers.

{\bs 1. Addition:}
$$
\{X_i\mid i = \OL{1, n}\} + \{Y_i\mid i = \OL{1, n}\} = \{X_i +
Y_i\mid i = \OL{1, n}\}.
$$

{\bs 2. Multiplication by numbers:} let $\la$ be a number; then
$$
\la\{X_i\mid i = \OL{1, n}\} = \{\la X_i\mid i = \OL{1, n}\}.
$$

In addition, we will be examining the element-by-element
multiplication of arrays by square matrices of the appropriate
dimensions.

{\bs 3. Left Multiplication by a Matrix:} let $K$ be a square
matrix of dimensions $p\times p$. Suppose
\begin{equation}\label{(1.2)}
K\{X_i\mid i = \OL{1, n}\} = \{KX_i\mid i = \OL{1, n}\}.
\end{equation}
Note that the multiplication of an array by a number can be examined
as a special case of multiplication by a square matrix.
Specifically: multiplication by the number $\la$ is multiplication by
the matrix $\la I$, where $I$ is the identity matrix of dimensions
$p \times p$.

{\bs 4. Right Multiplication by matrices: } let $Q = \|q_{ij}\mid i =
\OL{1, n}, j = \OL{1, n}\|$ --- a square $n$ by $n$ matrix.
Let us define the right multiplication of array $\BT$ \eqref{(1.1)}
by matrix $Q$ as
\begin{equation}\label{(1.3)}
\{X_i\mid i = \OL{1, n}\} Q = \{\SL_{j=1}^n X_j q_{ij}\mid i = \OL{1, n}\}.
\end{equation}
It is clear that the product $\BT Q$ is defined by the common
matrix multiplication method of ``row by column'' with the difference
that elements of a row of $\BT$ (array $\BT$) are not
numbers but columns $X_1, \ldots, X_n$.

{\bs 5.} Let us define the {\bs inner product} in array space.
For it's properties we shall call it the scalar product (or, generalized\
scalar product).
In more detail: let
$$
\BT = \{X_i\mid i = \OL{1, n}\},\quad \BR = \{Y_i\mid i = \OL{1, n}\}.
$$
\begin{definition}\label{scalar-def}
  The Scalar (generalized scalar) product of arrays
  $\BT$ and $\BR$ is defined as
  \begin{equation}\label{(1.4)}
    \p {\BT}{\BR} = \SL_{i=1}^n X_i Y_i^T.
  \end{equation}
\end{definition}
The result of the product is a square $p$ by $p$ matrix.
The scalar product is not commutative:
$$
\p {\BR}{\BT} = \p {\BT}{\BR}^T.
$$

{\bs 6.} The \textbf{Scalar square} of array
\begin{equation}\label{(1.5)}
\p {\BT}{\BT} = \SL_{i=1}^n X_i X_i^T.
\end{equation}
is a symmetric and non-negatively defined $\m pp$ matrix.
For the representation of the scalar square, we shall use the
traditional symbol of absolute value:
$\p {\BT}{\BT} =
|\BT|^2$.
In our case,
$|\BT|$ 
is the so-called matrix module. \cite{Horn} 

{\bs 7. The Properties of the scalar product} in array spaces
are similar to the properties of the traditional scalar product
in euclidean vector spaces. If
$\BT_1,
\BT_2, \BT_3$
are arrays in general form, then
\begin{align*}
&\p{\BT_1+\BT_2}{\BT_3} = \p{\BT_1}{\BT_3} + \p{\BT_2}{\BT_3};\\
&\p{K \BT_1}{\BT_2} = K \p{\BT_1}{\BT_2}
\text{ where $K$ is a square $\m{p}{p}$ matrix};\\
&\p {\BT_1}{\BT_1} \succcurlyeq 0
\text{ in the sense of the comparison of square symmetrical matrices;}\\
&\p {\BT_1}{\BT_1}=0
\mbox{ iff } \BT_1=0.
\end{align*}

{\bs 8.} We say that array $\BT$ is {\bs orthogonal} to
array $\BR$, if
$$
\p {\BT}{\BR} = 0.
$$
Note that if
%_-_п___''____, '_''__ _'__>__
$\p {\BT}{\BR} = 0$,
then also
% ''__ __
$\p
{\BR}{\BT} = 0$.
Therefore the property of orthogonality of arrays is reciprocal.
The orthogonality of arrays $\BT$ and $\BR$
shall be denoted as $\BT\perp \BR$.

{\bs 9.} Notice a Pythagorean theorem: if arrays $\BT$ and
$\BR$ are orthogonal, then
\begin{equation}\label{(1.6)}
\p {\mathbf{T+R}}{\mathbf{T+R}}=\p {\BT}{\BT}+\p {\BR}{\BR}.
\end{equation}
We note again that the result of a scalar product of two
arrays is a $\m{p}{p}$ matrix, therefore in array spaces square
matrices of corresponding dimensions should play the role of scalars.
In particular, left multiplication by a $\m{p}{p}$ matrix shall be
understood as multiplication by a scalar, and array $k\BT$
shall be understood as proportional to array $\BT$.

Together with arrays of the form \eqref{(1.1)} we shall consider
one-to-one corresponding matrices
\begin{equation}\label{(1.7)}
\es X = ||X_1, X_2, \ldots, X_n||.
\end{equation}
Matrix \eqref{(1.7)} is a matrix with $p$ rows and $n$ columns.

{\bs Notation.} Matrices with $p$ rows and $n$ columns shall be called
$\m{p}{n}$ matrices.  The set of $\m{p}{n}$ matrices we shall call
$\Mr^p_n$.  Matrices of dimensions $\m{p}{1}$ we shall call
$p$-columns, or simply columns.  The set of $p$-columns we represent
as $\Mr^p_1$.  Matrices $\m{1}{n}$ we shall call $n$-rows, or simply
rows.  The set of $n$-rows we represent as $\Mr^1_n$.

Many operations with arrays can be carried out in their matrix forms.
For instance, the addition of arrays is equivalent to the addition of
their corresponding matrices; left multiplication by a square
$\m{p}{p}$ matrix $k$ is equivalent to the matrix product $k \es X$;
right multiplication by matrix $Q$ is equivalent to the matrix product
$\es XQ$; the scalar product of arrays
$$
\BT_X = \{X_i\mid i = \OL{1, n}\},\quad
\BT_Y = \{Y_i\mid i = \OL{1, n}\}
$$
is equal to the product of their equivalent matrices $\es X$ and $\es Y$:
\begin{equation}\label{(1.8)}
\p {\BT_X}{\BT_Y} = \es X\es Y^T.
\end{equation}
We show, for instance, that array $\BT Q$ corresponds to matrix
$\es XQ$.
Here $\BT$ is the arbitrary array of form
\eqref{(1.1)} and $\es X$
is the corresponding $\m{p}{n}$ matrix
\eqref{(1.7)}.
Let
$Q = \{q_{\al\be}\mid \al, \be = \OL{1, n}\}$
be a $\m{n}{n}$ matrix (with numerical elements $q_{\al\be}$).
\begin{prop}
Matrix $\es XQ$ corresponds to array $\BT Q$.
\end{prop}
\begin{proof}
Elements of array $\BT$, being columns of matrix $\es X$, must be
represented in detailed notation.
Let
$$
X_j = (x_{1j}, x_{2j}, \ldots, x_{pj})^T,\ j = \OL{1, n}.
$$
In this notation,
\begin{equation}\label{(1.9)}
\es XQ = \|\SL_{j=1}^nx_{ij}q_{jk}\mid i = \OL{1, p}, k = \OL{1,
n}\|.
\end{equation}

The array
$$
\BT_Y = \{Y_k\mid k = \OL{1, n}\},
$$
corresponds to matrix $\es XQ$ where
$$
Y_k = (y_{1k}, y_{2k}, \ldots, y_{pk})^T,
$$
and
$$
y_{ik} = \SL_{j=1}^nx_{ij}q_{jk},
$$
by \eqref{(1.9)}.
Array $\BT Q$, by definition \eqref{(1.3)}, is equal to
$$
\BT Q = \{X_i\mid i = \OL{1, n}\}Q =
\{\SL_{j=1}^nX_jq_{kj}\mid k = \OL{1, n}\} = \{Z_k\mid k = \OL{1,
n}\},
$$
where $p$-row
\begin{equation}\label{(1.10)}
\begin{split}
Z_k &= \SL_{j=1}^nX_jq_{kj} = \SL_{j=1}^n(x_{1j}, \ldots,
x_{pj})^Tq_{kj} =\\ &= \left(\SL_{j=1}^nx_{1j}q_{kj},
\SL_{j=1}^nx_{2j}q_{kj}, \ldots, \SL_{j=1}^nx_{pj}q_{kj}\right)^T.
\end{split}
\end{equation}
Comparing expressions \eqref{(1.9)} and \eqref{(1.10)},
we see the equality of their elements.
\end{proof}

Thus in a tensor product $\s pn \otimes \s 1n $ we introduced the
structure of a module over the ring of square matrices supplied with
an inner product, which we called a scalar product.

\subsection{Linear Transformations}
Many concepts of classical linear algebra transfer to array space
almost automatically, with the natural expansion of the field of
scalars to the ring of square matrices.  For instance, the
transformation $f(\cdot)$ of array space \eqref{(1.1)} onto itself is
called {\it linear} if for any array $\BT_1$ and $\BT_2$ and for any
$\m{p}{p}$ matrix $k_1$ and $k_2$
\begin{equation}\label{(2.1)}
f(K_1\BT_1 + K_2\BT_2) = K_1f(\BT_1) + K_2f(\BT_2).
\end{equation}
Linear transformations in array space are performed by right
multiplication by square matrices.  Let $Q$ be an arbitrary $\m{n}{n}$
matrix, $\BT$ be an arbitrary array \eqref{(1.1)}.  That
transformation
$$
f(\BT) = \BT Q
$$
is linear in the sense of \eqref{(2.1)}, directly follows from the
definition \eqref{(1.3)}.  That there are no other linear
transformations follows from their absence even in the case $p = 1$.
(As we know, all linear transformations in vector spaces of rows are
performed by right multiplication by square $\m{n}{n}$ matrices.)

Note that the matrix form \eqref{(1.7)} of representing an array is
fitting also for the representation of linear transformations:
matrix $\es XQ$ (the product of matrices $\es X$ and $Q$) coincides
with the matrix form of an array \eqref{(1.3)}
$$
\BT Q = \{X_i\mid i = \OL{1, n}\}Q = \{\SL_{j=1}^nX_jq_{ij}\mid i = \OL{1,
n}\}.
$$

We shall call a linear transformation of array space onto itself
{\it orthogonal} if this transformation preserves the scalar
product. It means that for any arrays $\BT$ and $\BR$
$$
\p{\BT Q}{\BR Q} =\p{\BT}{\BR}.
$$
It is easy to see that orthogonal transformations are performed by
right multiplication by orthogonal matrices.
Indeed,
\begin{multline*}
\p{\BT Q}{\BR Q} =
\SL_{i=1}^n\left(\SL_{j=1}^nX_jq_{ij}\right)\left(\SL_{l=1}^nY_lq_{il}\right)^T
=\\
=\SL_{j=1}^n\SL_{q=1}^nX_jY_l^T\left(\SL_{i=1}^nq_{ij}q_{il}\right)
= \SL_{j=1}^nX_jY_j^T,
\end{multline*}
since matrix $Q$ is orthogonal and therefore
$$
\SL_{i=1}^nq_{ij}q_{il} = \delta_{jl}\quad \text{(Kronecker
symbol).}
$$

\subsection{Generating Bases and Coordinates}\label{Sec-basis}

Let $\al \in \Mr^p_1, x \in \Mr^1_n, \al x \in \Mr^p_n$. Here $\al
x$ denotes the product of matrices $\al$ and $x$. The matrices of
form $\al x$ plays a special role in array spaces.

%Matrices of dimensions $\m{p}{n}$ of form $\al x$, where $\al$ is a
%$p$-column and $x$ is a $n$-row-- the symbol $\al x$ represents the
%product of matrices $\al \in \Mr^p_1, x \in \Mr^1_n$ play a special
%role in array space.

Let $n$-rows $e_1, e_2, \ldots, e_n \in \Mr^1_n$ form the basis
of the space $\Mr^1_n$.
Let $\al_1, \al_2, \ldots, \al_n \in \Mr^p_1$ be arbitrary $p$-columns.
Let us consider $\m{p}{n}$-matrices $\al_1 e_1, \al_2 e_2, \ldots,
\al_n e_n$.

\begin{theorem}\label{th1}
Any $\m{p}{n}$-matrix $\es X$ \eqref{(1.7)} can be represented as
\begin{equation}\label{(3.1)}
\es X = \SL_{i=1}^n\al_ie_i
\end{equation}
for some choice of $\al_1, \al_2, \ldots, \al_n \in \Mr^p_1$
uniquely.
\end{theorem}
\begin{proof}
Let us define $\m{n}{n}$-matrix $E$ formed by $n$-rows
$e_1, e_2, \ldots, e_n$. Let us also introduce a $\m{p}{n}$-matrix
$A$ formed by $p$-columns
$\al_1, \al_2, \ldots, \al_n$.
With matrices $A$ and $E$ the sum \eqref{(3.1)} can be represented as
$$
\SL_{i=1}^n\al_ie_i = AE.
$$

Here are some calculations to confirm that assertion.
Let
$\al_i =
(\al_{1i}, \al_{2i}, \ldots, \al_{pi})^T, \ e_i = (e_{i1}, e_{i2},
\ldots, e_{in})$.
$$
\SL_{i=1}^n\al_ie_i = \SL_{i=1}^n(\al_{1i},
\al_{2i}, \ldots, \al_{pi})^T(e_{i1}, e_{i2}, \ldots, e_{in}).
$$
The element at $(k,l)$-position of each product $\al_i e_i,
i=1,\dots, n$,  is in essence
$\al_{ki}e_{il}$. Their total sum, which is the element of matrix
$\sum_{i=1}^n \al_i e_i$, is
$\sum_{i=1}^n\al_{ki}e_{il}$.

The element at $(k,l)$-position  of matrix $AE$  (calculated by the
row by column rule) is
$$
\SL_{i=1}^n\al_{ki}e_{il}.
$$
The calculated results are equal.

The theorem shall be proven if we show that the equation
\begin{equation}\label{(3.2)}
\es X = AE
\end{equation}
has a unique solution relative to the $\m{p}{n}$-matrix $A$.
Since the $\m{n}{n}$-matrix $E$ is invertible, the solution is
obvious:
\begin{equation}\label{(3.3)}
A = \es XE^{-1}.
\end{equation}
\end{proof}

The theorem allows us to say that the basis of $\Mr^1_n$ generates
the space $\Mr^p_n$ (using the above method).
Thus, the bases in $\Mr^1_n$ shall be called {\it generating bases}
in relation to $\Mr^p_n$.
The $p$-columns
$\al_1, \al_2, \ldots, \al_n$ from \eqref{(3.1)}  can be understood
as  the coordinates of $\es X$ in the generating basis $e_1, e_2,
\ldots, e_n$.
For the canonical basis of the space $\Mr^1_n$ (where $e_i$ is an
$n$-row, in which the $i$th element is one, and the others are zero)
coordinates $\es X$ relative to this basis are $p$-columns
$X_1, \ldots, X_n \in \Mr^p_n$, which form the matrix $\es X$.

The coordinates of the $\m pn$-matrix $\es X$ in two different
generating bases are connected by a linear transformation.
For example, let $n$-rows $f_1, \ldots, f_n$ form the basis in
$\Mr^1_n$.
Let $F$ be an $\m nn$-matrix composed of these $n$-rows.
By Theorem~\ref{th1} there exists a unique set of $p$-columns
$\be_1, \be_2, \ldots, \be_n$ that are coordinates of $\es X$
relative to the generating basis $f_1, \ldots, f_n$.
Matrices $B = ||\be_1, \ldots, \be_n||$ and $F$ are connected to the
$\m pn$-matrix $\es X$ by the equivalence
\begin{equation}\label{(3.4)}
\es X = BF.
\end{equation}
With  \eqref{(3.2)} this gives
$$
BF = AE.
$$
Therefore,
$$
B = AEF^{-1},\quad A = BFE^{-1}.
$$

\begin{corollary}
If the generating bases $e_1, \ldots, e_n$ and $f_1, \ldots, f_n$
are orthogonal, then the transformation of the coordinates of an
array in one basis to the coordinates of it in another is performed
through multiplication by an orthogonal matrix.
\end{corollary}

Let us consider an arbitrary orthogonal basis $e_1,\ldots, e_n$ in
$\Mr^p_n$.
For arbitrary $\m pn$-matrices $\es X$ and $\es Y$ we have the
decompositions of \eqref{(3.1)} with respect to this basis:
$$
\es X = \SL_{i=1}^n\al_ie_i,\quad \es Y = \SL_{i=1}^n\ga_ie_i.
$$
We can  express the scalar product of $\es X$ and $\es Y$
through their coordinates. It is easy to see that
\begin{equation}\label{(3.5)}
\p{\BT_X}{\BT_Y} = \es X\es Y^T= \SL_{i=1}^n\al_i\ga_i^T.
\end{equation}
\begin{corollary}
In an orthogonal basis, the scalar product of two arrays is equal to
the sum of the paired product of the coordinates.
\end{corollary}
\begin{proof}
Indeed,
$$
\es X\es Y^T= \p{\SL_{i=1}^n\al_ie_i}{\SL_{j=1}^n\ga_je_j} =
\SL_{i=1}^n\SL_{j=1}^n\al_ie_ie_j^T\ga_j^T =
\SL_{i=1}^n\al_i\ga_i^T,
$$
since for the orthogonal basis $e_i e_j^T = \delta_{ij}$.
\end{proof}

Therefore the scalar square of $\BT_X$ equals
$$
|\BT_X|^2 = \es X \es X^T = \SL_{i=1}^n\al_i\al_i^T.
$$
We can conclude from here that the squared length of an array is
equal to the sum of its squared coordinates in an orthogonal basis,
as for the squared euclidean length of a vector.

\subsection{Submodules}
We define a Submodule in array space \eqref{(1.1)} (or the space of
corresponding matrices \eqref{(1.7)}) to be a set which is closed
under linear operations: addition and multiplication by scalars.
Remember that multiplication by scalars means left multiplication by
$\m pp$-matrices.  For clarity, we shall discuss arrays in their
matrix forms in future.

\begin{definition}\label{def2}
The set $\ML \subset \Mr^p_n$ we shall define to be the submodule
of space $\Mr^p_n$, if for any
$\es X_1, \es X_2 \in \ML$
\begin{equation}\label{(4.1)}
K_1\es X_1 + K_2\es X_2 \in \ML
\end{equation}
with arbitrary $\m pp$-matrices $K_1, K_2$.
\end{definition}

\begin{theorem}\label{th2}
Any submodule $\ML$,\ $\ML \subset \Mr^p_n$, is formed by some
linearly independent system of $n$-rows.
The number of elements in this system is uniquely determined by
$\ML$. This number may be called the dimension of the linear
subspace~$\ML$.
\end{theorem}

\begin{proof}
Let $\es X \in \ML$.
The set of $\m pn$-matrices of the form $K\es X$ (where $K$ is an
arbitrary $\m pp$-matrix) forms a submodule.
Let us label it as $\ML(\es X)$
Let $x_1, \ldots, x_p$ be $n$-rows of the $\m pn$-matrix $\es X$.
Let us choose from among these $n$-rows a maximal linear independent
subsystem, such as $y_1, \ldots, y_k$.
It is obvious that
$$
\ML(\es X) = \{\es Y\mid \es Y = \SL_{i=1}^k\be_iy_i,\ \be_1,
\ldots, \be_k \in \Mr^p_1\}.
$$
If $\ML(\es X) = \ML$, then $y_1, \ldots, y_k$ form a generating
basis for $\ML \subset \Mr^p_n$.
If $\ML(\es X) \ne \ML$, then let us find in $\ML$ an element, say
$\es Z$, that does not belong to $\ML(\es X)$.
Let us expand the system $y_1, \ldots, y_k$ with $n$-rows $z_1, \ldots,\
 z_p$ of $\m pn$-matrix $\es Z$. Then we find in this set of $n$-rows the maximal
linearly independent subsystem, and repeat.
At some point the process ends.
\end{proof}

\begin{corollary}
Any submodule $\ML \subset \Mr^p_n$ can be expressed as the
sum of one-dimensional submodules $\ML_i \subset \Mr^p_n$:
\begin{equation}\label{(4.2)}
\ML = \ML_1\oplus\ML_2\oplus\ldots\oplus\ML_l,
\end{equation}
where
$$
\ML_i = \{\es X\mid \es X = \al y_i, \al \in \Mr^p_1\}
$$
for some $y_i \in \Mr^1_n$.
The number $l$ is the same in any representation \eqref{(4.2)} of
$\ML$.
This number can be called the dimension of subspace $\ML$: $l =
\dim\ML$.
\end{corollary}

{\bs Note.}
One can choose an orthogonal linearly independent system of $n$-rows
that generates $\ML$.
For proof, it is sufficient to note that the generating system
can be transformed into an orthogonal one by the process of
orthogonalization.
\smallskip

Theorem~\ref{th2} establishes the one-to-one correspondence between
linear subspaces of vector space $\Mr^1_n$ and the submodules
of the matrix space $\Mr^p_n$.
Let us state this as

\begin{corollary}
Each linear subspace $L$ in the space of $n$-rows $\Mr^1_n$
corresponds to some submodule $\ML$ in the space of $\m
pn$-matrices $\Mr^p_n$.
The dimensions of the linear subspace $L$ and the submodule $\ML$ coincide.
\end{corollary}
In this manner, the space $\Mr^p_n$ (and the corresponding array
space) and the space $\Mr^1_n$ have an equal ``supply'' of linear
subspaces and submodules.
This leads to significant consequences for multivariate
statistical analysis.

\begin{definition}
An orthogonal compliment of the submodule $\ML$ with
respect to the whole space is said to be
\begin{equation}\label{(4.3)}
\ML^{\perp} = \{\es X\mid \es X \in \Mr^p_n,\ \p{\es X}{\es Y} =
0,\ \forall\ \es Y \in \ML\}.
\end{equation}
\end{definition}s
It is easy to see that $\ML^{\perp}$ is a submodule and that
$$
\ML\oplus\ML^{\perp} = \Mr^p_n,\quad \dim\ML^{\perp} = n -
\dim\ML.
$$

\subsection{Projections onto Submodules}\label{sec1.5}

Let us consider array space \eqref{(1.1)} with the introduced scalar
product \eqref{(1.4)}.
Let $\ML$ be the submodule \eqref{(4.1)}.
Let us call {\it the projection of array $\BT$ onto a linear
subspace} $\ML$ the point of $\ML$ that is closest to $\BT$ in the
sense of comparing scalar squares \eqref{(1.5)}.

Let us say it in details.
Let array $\BR$ pass through the set $\ML$.
We shall call the point $\BR^0 \in \ML$ {\it closest to $\BT$} if
for any $\BR \in \ML$
$$
\p{\BT - \BR^0}{\BT - \BR^0}
\preccurlyeq \p{\BT - \BR}{\BT - \BR}.
$$
Note that $\p{\BT - \BR}{\BT - \BR}$ is the function of $\BR$ with
values in the set of $\m pp$-matrices.
The existence of a minimal element in the set of matrices (generated
by $\BR \in \ML$) is not obvious and is not provided naturally.
So the existence of $\proj_{\ML}\BT$ needs to be proved.
We state this result in the following theorem.

\begin{theorem}\label{th3}
The projection of $\BT$ onto $\ML$ exists, is unique, and has
the expected (euclidean) properties.

\begin{enumerate}
\item For any array $\BR \in \ML$,
$$
|\BT - \BR|^2 \succcurlyeq |\BT -
\proj_{\ML}\BT|^2,
$$
with equality iff $\BR = \proj_{\ML}\BT$;
\item $(\BT - \proj_{\ML}\mathbf{T })\perp \ML$;
\item $\proj_{\ML}(K_1\BT_1 + K_2\BT_2) = K_1\proj_{\ML}\BT_1 +
K_2\proj_{\ML}\BT_2$.
\end{enumerate}
\end{theorem}

\begin{proof} 
Let $\es X \in \Mr^p_n$ be an arbitrary $\m pn$-matrix.
As was shown, any submodule $\ML \subset \Mr^p_n$ is
equivalent to a linear subspace $L$ in the space of $n$-rows, $L
\subset \Mr^1_n$.
Let $\Pi$ be a projection matrix onto $L$ in the space $\Mr^1_n$,
that is, for any $x \in \Mr^1_n$
$$
\proj_{\ML}x = x\Pi.
$$

To prove the theorem we need the following Lemma~\ref{lemma1} and
Theorem~\ref{th4}.

\begin{lemma}\label{lemma1}
Let $\ML \subset \Mr^p_n$ be a submodule in the  space of $\m
pn$-matrices, and let $L \subset \Mr^1_n$ be a linear subspace in
the space of $n$-rows which generates $\ML$.
Then for any $\la \in  \Mr^p_1$
$$
\la^T\ML = L.
$$
\end{lemma}

\begin{proof}[Proof of Lemma]
Let $r = \dim L, r \le n$.
Let us choose within $L$ the basis $e_1, \ldots, e_r$.
As we know, the subspace $\ML \in \Mr^p_n$ can be represented as
$$
\ML = \{\es Y\mid \es Y = \SL_{k=1}^r\al_ke_k, \al_1, \ldots,
\al_k \in \Mr^p_1\}.
$$
Let $\es Y \in \ML$, then for some
$\al_1, \ldots, \al_k \in \Mr^p_1$
$$
\es Y = \SL_{k=1}^r\al_ke_k.
$$
Therefore, under any $\la \in \Mr^p_1$
$$
\la^T\es Y = \SL_{k=1}^r(\la^T\al_k)e_k \in L,
$$
since $\la^T\al_1, \ldots, \la^T\al_r$ are numerical coefficients.
\end{proof}

\begin{theorem}\label{th4}
Let $\ML$ be a submodule in the space of $\m pn$-matrices.
Let $L$ be a linear subspace in the space of $n$-rows which generates
$\ML\subset \mathbb R^p_n$.
Let $\Pi$ be a projection $\m nn$-matrix onto $L$, that is, for any
$x \in \Mr^1_n$
$$
\proj_{L}x = x\Pi.
$$
Then for any $\es X \in \Mr^p_n$
$$
\proj_{\ML}\es X =\es X\Pi.
$$
\end{theorem}
\begin{proof}
We must show that for any $\es Y \in \ML$
\begin{equation}\label{(5.1)}
\p{\es X - \es Y}{\es X - \es Y} \succcurlyeq \p{\es X - \es X\Pi}
{\es X - \es X\Pi}
\end{equation}
with equality if and only if $\es Y = \es X\Pi$.
The inequality  between two symmetrical $\m pp$-matrices in
\eqref{(5.1)} means that for any
$\la \in \Mr^p_1$
$$
\la^T\p{\es X - \es Y}{\es X - \es Y}\la\, \ge \la^T\p{\es X - \es
X\Pi} {\es X - \es X\Pi}\la,
$$
thus
$$
|\la^T(\es X - \es Y)|^2 \ge |\la^T(\es X - \es X\Pi)|^2,
$$
thus
$$
|\la^T\es X - \la^T\es Y|^2 \ge |\la^T\es X - (\la^T\es X)\Pi|^2.
$$
As was noted above, the $n$-row $y = \la^T \es Y$ belongs to $L$
and $\la^T\es X\Pi =x\Pi$ is a projection of $\la^T\es X$ onto $L$.
Due to the properties of euclidean projection, we get for any $y \in
L$
$$
|x - y|^2 \ge |x - x\Pi|^2
$$
with equality if and only if
$y = x\Pi$. Thus, $\es X\Pi$ is the nearest point to $\es X$ in
$\ML$.%, that is, $\es X\Pi$ is the unique projection of $\es X$ onto
%$\ML$.
\end{proof}

Now we return to proving Theorem \ref{th3}. From Theorem
\ref{th4} we know $\es X\Pi$ is the unique projection of $\es X$
onto $\ML$. So, statement 1 of Theorem \ref{th3} is proven.

The explicit expression $\proj_{\ML}\es X=\es X\Pi$ confirms that
the operation of projection onto a submodule is a linear operation.
So, statement 3 of Theorem \ref{th3} is proven as well.

To complete the proof of Theorem \ref{th3} we need to show statement
2. Let $e_1,\ldots,e_r$ be an orthogonal basis of $L$ and
$e_{r+1},\ldots,e_n$ be an orthogonal basis of $L^\bot$.
Then, $e_1,\ldots,e_n$ is the orthogonal basis of $\mathbb R^1_n$.
%Then the statement $(\es X - \es X\Pi) \perp \ML$ is provably quite
%easily:
In this orthogonal basis, if
$$
\es X = \SL_{i=1}^n\al_ie_i,
$$
then
$$
\es X - \es X\Pi = \SL_{i=r+1}^n\al_ie_i.
$$
Since $\es Y \in \ML$,
$$
\es Y = \SL_{i=1}^r\be_ie_i
$$
for some $\be_1, \ldots, \be_r \in \Mr^p_1$.
Therefore:
\begin{align*}
(\es X - \es X\Pi)\es Y^T &= \p{\SL_{i=r+1}^n\al_ie_i}
{\SL_{i=1}^r\be_ie_i} = \SL_{i=r+1}^n\SL_{j=1}^r \p{\al_ie_i}
{be_je_j} =\\
&= \SL_{i=r+1}^n\SL_{j=1}^r\al_ie_i(\be_je_j)^T =
\SL_{i=r+1}^n\SL_{j=1}^r\al_ie_ie_j^T\be_j^T = 0,
\end{align*}
since $e_i e_j^T = 0$ when $i \ne j$.
\end{proof}

\subsection{Matrix Least Squares Method}\label{sec1.6}

Calculating projections onto a submodule
$\ML \subset \Mr^p_n$ become easier if the form of projection onto
the linear subspace $L$ which generates $\ML$ is known.
By the lemma from Section \ref{sec1.5}, for any $\la \in \Mr^p_1$
\begin{equation}\label{(6.1)}
\la^T\proj_{\ML}\es X = \proj_L(\la^T\es X).
\end{equation}
Assume that for the right part of \eqref{(6.1)} we have an explicit
formula $y = \proj_L x$.
Then because of the linearity  this gives  us for
$\proj_{\ML}(\la^T\es X)$ an explicit formula $\la^T\es Y$.
Therefore
\begin{equation}\label{(6.2)}
\la^T\proj_{\ML}\es X = \la^T\es Y.
\end{equation}
So we get an explicit expression for
$\proj_{\ML}\es X$. One can say this is the calculation of
$\proj_{\ML}\es X$ by Roy's method. \cite{Roy} 

\bigskip\noindent {\bs
Example: calculating the arithmetic mean.
} Let $X_1, X_2,$ $\ldots, X_n \in \Mr^p_1$ be the set of
$p$-columns.
Let us consider the array $\BT = \{X_i\mid i=\OL{1,n}\}$ and
represent it in matrix form.
\begin{equation}\label{(6.3)}
\es X = \|X_1, X_2, \ldots, X_n\|.
\end{equation}
Our task is to find the array $\es Y$ with identical columns, i.e.,
an array of form
\begin{equation}\label{(6.4)}
\es Y = \|Y,Y,\dots,Y\|,\ Y \in \Mr^p_1,
\end{equation}
closest to \eqref{(6.3)}.
Arrays of form \eqref{(6.4)} produce a one-dimensional submodule.
We shall denote it by
$\ML,\ \ML \subset \Mr^p_n$. We have to find $\proj_{\ML}\es X$.
The submodule $\ML$ is generated by a one-dimensional linear subspace 
$L,L \subset \Mr^1_n$, spanned by $n$-row
$e = (1, 1, \ldots, 1)$.

Let $x$ be an arbitrary $n$-row, $x = (x_1, \ldots, x_n)$.
The form of projection of $x$ onto $L$ is well known:
$$
\proj_{L}x = (\OL{x}, \ldots, \OL{x}).
$$
Applying Roy's method to matrix $\es X$ \eqref{(6.3)} we get over
to $n$-row $x = \la^T\es X$, where $x_i=\la^T X_i,\ i=\OL{1,n}$.
It is then clear that
$$
\proj_Lx = (\la^T\OL{X}, \ldots, \la^T\OL{X}).
$$
Therefore,
\begin{equation}\label{(6.5)}
\proj_{\ML}\es X = (\OL{X}, \ldots, \OL{X}).
\end{equation}
Of course, this is not the only and not always the most efficient
method.
In this example, like in other cases, one can apply the matrix method
of least squares and find
\begin{equation}\label{(6.6)}
\hat{Y} = \arg\min\limits_{Y\in\Mr^p_1}\SL_{i=1}^n(X_i - Y)(X_i -
Y)^T.
\end{equation}

{\bs Solution.}
Let us transform the function in \eqref{(6.6)}: for any
$Y\in\mathbb R^p_1$
\begin{multline}\label{(6.7)}
\SL_{i=1}^n(X_i - Y)(X_i - Y)^T = \SL_{i=1}^n[(X_i - \OL{X}) +
(\OL{X} - Y)][(X_i - \OL{X}) + (\OL{X} - Y)]^T = \\ =
\SL_{i=1}^n(X_i - \OL{X})(X_i - \OL{X})^T + n\SL_{i=1}^n(\OL{X} -
Y)(\OL{X} - Y)^T = (1) + (2),
\end{multline}
since ``paired products'' \ turn to zero:
$$
\SL_{i=1}^n(X_i - \OL{X})(\OL{X} - Y)^T = 0,\qquad
\SL_{i=1}^n(\OL{X} - Y)(X_i - \OL{X})^T = 0.
$$
Now the  function \eqref{(6.7)} is a sum of two nonnegatively
defined matrices, and the first one does not  depend on  $Y$.
The minimum attains at $Y=\OL{X}$: at that point the nonnegatively
definite matrix (2) turns to zero.

The answer is an arithmetic mean, that is,
$$
\hat{Y} = \OL{X}. %\eqno(6.8)
$$
Of course, it is well known.
It can be find by applying not the matrix but the ordinary method of
least squares:
$$
\hat{Y} = \arg\min\limits_{Y\in\Mr^p_1}\SL_{i=1}^n(X_i - Y)^T(X_i
- Y).
$$

The results of the matrix method similarly relate to the traditional
in the case of projection on other submodules $\ML \subset \Mr^p_n$.
The reason is simple: if an array $\es Y$ is the solution of a matrix
problem
$$
\SL_{i=1}^n(X_i - Z_i)(X_i - Z_i)^T = (\es X - \es Z)( \es X - \es
Z)^T \to \min\limits_{\es Z \in \ML},
$$
then $Y$ is a solution of the scalar problem as well,
$$
\tr\{\SL_{i=1}^n(X_i - Z_i)(X_i - Z_i)^T\} = \SL_{i=1}^n(X_i -
Z_i)^T(X_i - Z_i) \to \min\limits_{\es Z \in \ML}.
$$
Thus, for instance, in calculating the projection on submodulqes one
can use the traditional scalar method of least squares.  Both least
squares methods in linear models give us the same estimates of
parameters.  The necessity of matrix scalar products and the matrix
form of orthogonality, projection, submodules, etc becomes apparent in
testing linear hypothesis.  We shall relate this in the next section.

%%%%%%%%%%%%%%%%%%%%%%%%%%%%%%%%%%%%%%%%%%%%%%%%%%%%%%%%%%%%%%%%%%%%%%%%%%%%%%%%%%%%%%%%%%%%%%%%%%%%%%%%%%%%%%%%%%%%%%%%%

% \newpage
\section{Multivariate Linear Models}

\subsection{Arrays with Random Elements}
\label{p21}

Let us consider array \eqref{p21-11}, the elements of which are
$p$-dimensional random variables presented in the form of
$p$-columns.
\begin{equation}\label{p21-11}
\BT=\{X_i\mid i=\overline{1,n}\}.
\end{equation}
Remember that we treat such an array as a row composed of $p$-columns
under algebraic operations.  For arrays of form \eqref{p21-11} with
random elements, let us define mathematical expectation and
covariance.  The array
\begin{equation}\label{p21-12}
\e {\BT_X} = \{ \e {X_i} \mid i=\overline{1,n}\}
\end{equation}
is called the {\it mathematical expectation} of $\BT$.
We define the  covariance matrix of array \eqref{p21-11}
much like the covariance matrix of random vector.
Let
$$
t=(x_1,\dots,x_n)
$$
be an $n$-row composed of random variables $x_1,x_2,\dots,x_n$.  As we
know, the covariance matrix $\var t$ of the random vector $t$ is an
$\m n n$-matrix with elements
$$
\sigma_{ij}=\cov{x_i}{x_j},\quad
\text{where }
i,j=\overline{1,n}.
$$
Algebraically, with the help of matrix operations, the covariance matrix
of the random vector $t$ can be defined as
\begin{equation}\label{p21-13}
\var t=\e (t-\e t)^T(t-\e t).
\end{equation}
Following \eqref{p21-13}, we define the \textit{covariance array} of
random array \eqref{p21-11} as
\begin{equation}\label{p21-14}
\var \BT :=\e (\BT-\e \BT)^T(\BT-\e \BT)
=\{\cov{X_i}{X_j}\mid i,j=\overline{1,n}\}.
\end{equation}
Here $\cov{X_i}{X_j}$ is a covariance matrix of random column-vectors
$X_i$ and $X_j$,
\begin{equation}\label{p21-15}
\cov{X_i}{X_j}=\e(X_i-\e X_i)(X_j-\e X_j)^T.
\end{equation}
Note that we consider $\var \BT$ \eqref{p21-14} as a square
array of dimensions \m n n, the elements of which are \m pp- matrices
\eqref{p21-15}.

Let us consider the array $\BR$, obtained by the linear
transformation of array $\BT$ \eqref{p21-11}
\begin{equation}\label{p21-16}
\BR=\BT Q,
\end{equation}
where $Q$ is a \m nn-matrix.

It is clear that 
$$
\e \BR=(\e \BT) Q,
$$
\begin{equation}\label{p21-17}
\var \BR = \e [(\BT Q-\e \BT Q)^T(\BT Q-\e \BT Q)]=Q^T(\var
\BT)Q.
\end{equation}

In mathematical statistics, arrays with statistically independent
random elements are of especial interest when the covariance matrices
 of these elements are the same.
Let $\BT$ \eqref{p21-11} be an array such that
\begin{equation}\label{p21-18}
\cov{X_i}{X_j}=\delta_{ij}\Sigma,\quad  i,j=\overline{1,n}.
\end{equation}
Here $\Sigma$ is a nonnegatively defined \m pp-matrix and $\delta_{ij}$
is the  symbol of Kronecker.
Let us consider an orthogonal transformation of array $\BT$:
\begin{equation}\label{p21-19}
\BR=\BT C,
\end{equation}
where $C$ is an orthogonal \m nn-matrix.
The following lemma is fairly simple but important.

\begin{lemma}\label{lemma2}
\begin{equation}\label{p21-10}
\var \BR=\var \BT=\{\delta_{ij}\Sigma\mid
i,j=\overline{1,n}\}
\end{equation}
\end{lemma}
This lemma generalizes for the multivariate case the well-known 
property of spherical normal distributions.

\begin{proof}
The proof of the lemma is straightforward.
To simplify the formulas, assume that $\e \BT=0$.
Then,
\eqref{p21-17},
\begin{multline*}
\e (\BT C)=\e [(\BT C)^T(\BT C)]=C^T(\var
\BT)C=\\= C^T\{\delta_{ij}\Sigma\mid
i,j=\overline{1,n}\}C=\{\delta_{ij}\Sigma\mid
i,j=\overline{1,n}\}.
\end{multline*}
\end{proof}
% Unfortunately, we have no algebraic operations supplied which would
% immediately give the array \eqref{p21-10} in the multiplication of
% an \m nn identity matrix, \textit{treated as an array}, by matrix
% $\Sigma$.  With such an operation, this would be easy: there would
% be a complete analogy in formulation and proof to analogous
% assertions on random vectors.

Earlier, while discussing generating bases and coordinates (Section
\ref{Sec-basis}), we established that the transformation from the
coordinates of array $\BT$ in an orthogonal basis to
coordinates of this array in another basis can be done through
multiplication by an orthogonal matrix.
Therefore if the coordinates of some array in one orthogonal
basis are not correlated and have a common covariance matrix, then
the coordinates of the given array hold these properties in any
orthogonal basis.
From the remark above and just established Lemma \ref{lemma2} follows
\begin{lemma}\label{lemma3}
If the coordinates of a random array in an orthogonal basis are
uncorrelated and have a common covariance matrix, then the
coordinates of this array are uncorrelated and have the same common
covariance in any orthogonal basis.
\end{lemma}

This property is of great importance in studying linear statistical
models.

Finally, let us note that in introducing and discussing covariance arrays of
random arrays we have to work with the arrays themselves
\eqref{(1.1)} and not with the matrices \eqref{(1.7)} representing them.
%The matrix form of arrays does not work for that.

\subsection{Linear Models and Linear Hypotheses}

\begin{definition}
One says that array $\BT$ \eqref{p21-11} with random elements obeys a
linear model if
\begin{itemize}
\item[a)]
for some given submodule $\ML$
\begin{equation}\label{p2-21}
\e\BT\in\ML;
\end{equation}
\item[b)]
elements $X_1, \ldots, X_n$ of array $\BT$ are independent and
identically distributed.
\end{itemize}
\end{definition}
If this is common for all $X_i$, with $i=\overline{1,n}$ a gaussian
distribution, then we say that array $\BT$ follows a \textit{linear
gaussian model}.
We will now study linear gaussian models.

We shall denote with $\Sigma$ the common covariance matrix for all $p$-columns.
The array $\e \BT$ and matrix $\Sigma$ are \textit{parameters of the
model}.
They are generally unknown; although, $\Sigma$ is assumed to be
nondegenerate.

For random arrays following the gaussian model, \textit{linear hypotheses}
are often discussed.
Within the framework of the linear model \eqref{p2-21} the linear hypothesis
holds the form:
\begin{equation}\label{p2-22}
\e \BT\in\ML_1,
\end{equation}
where $\ML_1$ is a given submodule, and $\mathcal{L}_1
\subset \mathcal{L}$.

Let us show that the linear models and linear hypotheses discussed
in multivariate statistical analysis have the structure of \eqref{p2-21}
and \eqref{p2-22}.
The main linear models are factor and regression.
For example, let us consider the one-way layout and regression
models of multivariate statistical analysis.

The \textbf{One-way layout model} is the simplest of the ``analysis of
variance'' models.  It is a shift problem of several (say, $m$) normal
samples with identical covariance matrices.  The array of observations
in this problem has to have double numeration:
\begin{equation}\label{p2-23}
\BT=\{X_{ij}\mid j=\overline{1,m},
i=\overline{1,n_j}\}.
\end{equation}
Here $m$ is the number of different levels of the factor, which
affects the expected values of the response. Here, $n_j$ is the number
of independently repeated observations of the response on the level
$j$ of the factor, $j=\overline{1,m}$.  Finally, multivariate
variables $X_{ij}$ are independent realizations of a $p$-dimensional
response, $X_{ij}\in\s p 1$.  Assume $N=n_1+\cdots+n_m$.  The main
assumption of the model is: $X_{ij}\sim N_p(a_j,\Sigma)$.

We shall linearly order the observations which constitute the array
\eqref{p2-23} and then represent \eqref{p2-23} as a \m pN-matrix.
\begin{equation}\label{p2-25}
\es{X}=\|X_{11},X_{12},\dots,X_{1n_1},X_{21},\dots,X_{2n_2},X_{m1},\dots
X_{mn_m}\|.
\end{equation}
Note that
\begin{equation}\label{p2-26}
\e \es X=\|\underbrace{a_1,\dots,a_1}_{n_1 \text{ times}},
\underbrace{a_2,\dots,a_2}_{n_2 \text{ times}},\dots
,\underbrace{a_m,\dots,a_m}_{n_m \text{ times}}\|.
\end{equation}
Let us introduce $N$-rows
\begin{equation}\label{p2-27}
\begin{split}
&e_1=(\underbrace{1,\dots,1}_{n_1 \text{ times}},0,\dots,0),\\
&e_2=(\underbrace{0,\dots,0}_{n_1 \text{ times}},\underbrace{1,\dots,1}_{n_2 \text{ times}},0,\dots,0),\\
&\dots\\
&e_m=(\underbrace{0,\dots,0}_{n_1 \text{times}},
\underbrace{0,\dots,0}_{n_2 \text{ times}},
\dots,\underbrace{1,\dots,1}_{n_m \text{ times}}).
\end{split}
\end{equation}
It is obvious that
$$
\e\es X=\sum_{i=1}^m a_i e_i.
$$
Therefore $\e\es X$ belongs to an $m$ dimensional submodule of
the space \s p N spanned by $n$-rows \eqref{p2-27}.

The hypothesis $H_0:a_1=a_2=\cdots =a_m$, with which one usually
begins the statistical analysis of $m$ samples, is obviously a linear
hypothqesis in the sense of \eqref{p2-22} $H_0:\e\es X\in\ML_1$, where
$\ML_1$ is a one dimensional linear subspace spanned by the single
$N$-row $e=e_1+\dots+e_m$.

\textbf{Multivariate Multiple Regression in matrix form} is
\begin{equation}\label{ex1}
\es Y=\es A\es X+\es E,
\end{equation}
where $\es Y=\|Y_1,Y_2,\dots,Y_n\|$.  Here $\es Y$ is theq observed \m
pn-matrix of $p$-dimensional response; $\es X$ is a given design \m
mn-matrix; $\es A$ is a \m pm-matrix of unknown regression
coefficients; $\es E=\|E_1,E_2,\dots,E_n\|$ is a \m pn-matrix composed
of independent $p$-variate random errors $E_1,E_2,\dots,E_n$.  In
gaussian models
$$
E_i\sim N_p(0,\Sigma),
$$
where \m pp matrix $\Sigma$ is assumed to be non-degenerate.
Generally $\Sigma$ is believed to be unknown.

Let $A_1,A_2,\ldots,A_m$ be the $p$-columns forming matrix $\es A$;
let $x_1,x_2,\ldots,x_m$ be $n$-rows, forming matrix $\es X$.
Then
\begin{equation}\label{ex2}
\es A\es X=\sum_{i=1}^m A_i x_i.
\end{equation}
The resulting expression~\eqref{ex2} shows that $\e\es Y=\es A\es X$
belongs to an $m$ dimensional submodule of the space \s pn,
generated by the linear system of $n$-rows
$x_1,x_2,\dots,x_m$.

\subsection{Sufficient Statistics and Best Unbiased Estimates}
\label{p23}

Let us consider a linear gaussian model \eqref{p2-21}
in matrix form
\begin{equation}\label{p2-31}
\es X=\es{M}+\es{E}.
\end{equation}
where $\es{M}=\e\es X$ is an unknown \m pn-matrix;
$$
\es M=\|M_1,M_2,\dots,M_n\|\in\ML,
$$
where $\ML$ is a submodule of
\s pn;
$$
\es{E}=\|E_1,E_2,\dots,E_n\|
$$
is a \m pn-matrix, the $p$-columns $E_1,E_2,\dots,E_n$ of which are the
independent $N_p(0,\Sigma)$ random variables.

The unknown parameter of this gaussian model is a pair $(\es M, \Sigma)$.
Let us find sufficient statistics for this parameter using
the factorization criterion.

A likelihood of the pair $(\es{M}, \Sigma)$ based on $\es X$ is
\begin{equation}\label{p2-32}
\begin{split}
\prod_{i=1}^n &\left( \frac{1}{\sqrt{2\pi}} \right)^p
\frac{1}{\sqrt{\det \Sigma}} \exp{ \{ -\frac12 (X_i-M_i)^T \Sigma^{-1}(X_i-M_i)
\} }=\\
=&\left( \frac{1}{\sqrt{2\pi}} \right)^{np}
\left(\frac{1}{\sqrt{\det \Sigma}}\right)^n
\exp{ \{ -\frac12 \mathop{\mathrm{tr}} \Sigma^{-1} [\sum_{i=1}^n (X_i-M_i)(X_i-M_i)^T] \} }.
\end{split}
\end{equation}

The sum in square brackets is $\p{\es X-\es{M}}{\es X-\es M}$.
Let us represent $\es X-\es M$ as
$$
\es X-\es M=(\x-\pr L X)+(\pr L X -\es M)=(1)+(2)
$$
and note that
$$
(1)=\pr {L^\perp}{X}\in\ML^\perp,\quad (2)\in\ML.
$$
Thus, (Pythagorean Theorem)
\begin{equation}\label{ss2}
\begin{split}
\p{\x-\es M}{\x-\es M}=\p{\pr {L^\perp}{X}}{\pr {L^\perp}{X}}+
\p{\pr L X-\es M}{\pr L X-\es M}.
\end{split}
\end{equation}
We conclude that the likelihood \eqref{p2-32} is expressed
through the statistics $\pr L X$ and $\p{\pr
{L^\perp}{X}}{\pr{L^\perp}{X}}$, which are sufficient for $\es M,
\Sigma$.

The statistic $\pr L X$ is obviously an unbiased estimate of $\es M$.
As a function of sufficient statistics it is the best unbiased
estimate of $\es M$.  We can show that the best unbiased estimate of
$\Sigma$ is the statistic
\begin{equation}\label{p2-34}
\frac{1}{\dim {\mathcal {L^\perp}}}
\p{\pr {L^\perp}{X}}{\pr {L^\perp}{X}}
\end{equation}
after proving the following theorem
\ref{t5}.

\subsection{Theorem of Orthogonal Decomposition}
\label{p24}

\begin{theorem}\label{t5}
Let $\x=\|X_1,X_2,\dots,X_n\|$ be a gaussian \m pn matrix with independent
$p$-columns $X_1,X_2,\dots,X_n\in\s p1$, and $\var X_i=\Sigma$ for all
$i=1,\dots,n$.
Let $\ML_1,\ML_2,\dots$ be pairwise orthogonal submodules \s pn,
the direct sum of which forms \s pn:
$$
\s pn=\ML_1\oplus \ML_2\oplus\dots
$$
Let us consider the decomposition of \m pn-matrix $\es X$ into the
sum of orthogonal projections $\es X$ on the submodules
$\mathcal{L}_1,\mathcal{L}_2,\dots$:
$$
\x=\prj{\ML_1}{X}+\prj{\ML_2}{X}+\dots
$$
Then:
\begin{enumerate}
\item[a)]
random \m pn-matrices
$\prj{\ML_1}{X},\prj{\ML_2}{X},\dots$
are independent, normally distributed, and
$\e\prj{\ML_i}{X}=\prj{\ML_i}{\e \x}$;
\item[b)] $\p{\prj{\ML_i}{X}}{\prj{\ML_i}{X}}=W_p(\dim
{L_i},\Sigma,\Delta_i)$,
where $W_p(\nu,\Sigma,\Delta)$ indicates a random matrix (of size $\m pp$),
distributed under Wishart, with $\nu$ degrees of freedom and the parameter
of non-centrality $\Delta$.
In this case
$$\Delta_i=\p{\prj{\ML_i}{\e X}}{\prj{\ML_i}{\e X}}.$$
\end{enumerate}
\end{theorem}

\begin{proof}
Each submodule $\ML\subset\s pn$ has a one-to-one correspondence to
some linear subspace $L\subset\s 1n$ which generates it, and $\dim
\ML=\dim L$.  Let submodules $\ML_1,\ML_2,\dots\subset\s pn$
correspond to the subspaces $L_1, L_2,\dots\subset\s 1n$.  The
subspaces $L_1, L_2,\dots\subset\s 1n$ are pairwise orthogonal, and
their direct sum forms the entire space \s 1n.  Let us denote the
dimensions of submodules $\ML_1,\ML_2, \dots\subset\s pn$ (and
subspaces $L_1, L_2,\dots\subset\s 1n$) by $m_1,m_2,\dots$.

Let us choose in every subspace $L_1, L_2,\dots$  an orthogonal basis.
For $\ML_1$ let it be the $n$-rows $f_1,\dots,f_{m_1}$; for
$\ML_2$, the $n$-rows $f_{m_1+1},\dots, f_{m_1+m_2}$ etc.
With the help of these $n$-rows each of the submodules
$\ML_1,\ML_2,\ldots$ can be represented as the direct sum of
one dimensional submodules from \s pn.
For example, $\ML_1=\mathcal F_1\oplus\mathcal F_2\oplus\dots\oplus
\mathcal F_{m_1}$, where
\begin{align*}
&\mathcal F_1=\{ \y \mid \y=\alpha f_1, \ \alpha
\in\mathbb R^p_1
\} ,\\
&\mathcal F_2=\{ \y \mid \y=\alpha f_2, \ \alpha
\in\mathbb R^p_1
\} ,\\
&\cdots\\
&\mathcal F_{m_1}=\{ \y \mid \y=\alpha f_{m_1},\ \alpha
\in\mathbb R^p_1
\}.
\end{align*}
The set of all $n$-rows $f_1,f_2,\dots,f_n$ forms an orthogonal
basis in \s 1n and so does the generating basis in \s pn.
Therefore any \m pn-matrix $\x\in\s pn$ can be represented in the form
$$
\x=\sum_{i=1}^n Y_i f_i,
$$
where $Y_1,\dots,Y_n$ are some $p$-columns, that is
$Y_1,\dots,Y_n\in\s p1$, and
\begin{align*}
&\prj{\ML_1}{X}=\sum_{i=1}^{m_1} Y_i f_i,\\
&\prj{\ML_2}{X}=\sum_{i=m_1+1}^{m_2} Y_i f_i \qquad
\mbox{etc.}
\end{align*}
Here $p$-columns $Y_1,Y_2,\ldots,Y_n$ are coordinates of a \m pn-matrix \x{}
relative to the generating basis
$f_1,\dots,f_n$, while
the $p$-columns $X_1,X_2,\dots$, $X_n$ are coordinates of the same
\m pn-matrix \x{} relative to the orthogonal canonical basis \s 1n:
$e_1=(1,0,\dots)$, $e_2=(0,1,0,\dots)$
etc.
As was noted earlier (see Lemma \ref{lemma3}), the transformation from
some coordinates to others is performed through the right
multiplication of an \m pn-matrix \x{} by some orthogonal transformation \m nn-matrix, say by \m nn-matrix
$C$:
$$
\|Y_1,Y_2,\dots,Y_n\|=\|X_1,X_2,\dots,X_n\| C,
\quad\mbox{or}\quad \y=\x C.
$$
Thus the $p$-columns $Y_1,\dots,Y_n$ are mutually normally distributed.
Following  Lemma \ref{lemma3},
$$
\var \es Y = \var \es X=\{\delta_{ij}\Sigma\mid i,j=\overline{1,n}\}.
$$
This means that $Y_1,\dots,Y_n$ are independent gaussian $p$-columns with common covariance matrix $\Sigma$, just like
the $p$-columns $X_1,\dots,X_n$.

Let us consider random \m pp-matrices
$$
\p{\prj{\ML_1}{X}}{\prj{\ML_1}{X}}, \p{\prj{\mathcal
L_2}{X}}{\prj{\ML_2}{X}},\dots
$$
For example,
$$
\p{\prj{\ML_1}{X}}{\prj{\ML_1}{X}}=\sum_{i=1}^{m_1}
Y_i Y_i^T.
$$
The distribution of such random matrices is called a Wishart distribution.
If $\e Y_1=\e Y_2=\dots=\e Y_{m_1}=0$, we get the so-called central
Wishart distribution $W_p(m_1,\Sigma)$.
Let us note that if one uses the notation $W_p(m,\Sigma)$ for
a random matrix itself, not only for its distribution, then
one can say that
$$
W_p(m,\Sigma)=\Sigma^{\frac 12} W_p(m,I) \Sigma^{\frac 12},
$$
if one represents as $\Sigma^{\frac 12}$ a symmetric matrix,
the unique symmetric solution of the matrix equation:
$Z^2=\Sigma$.

One says that a random \m pp-matrix $W$ has the noncentral
Wishart distribution if
$$
W=\sum_{i=1}^m(\xi_i+a_i)(\xi_i+a_i)^T,
$$
where the $p$-columns $\xi_1,\xi_2,\ldots,\xi_m$ are iid
$N_p(0,\Sigma)$, $a_1,a_2,\dots,a_m$ are some nonrandom $p$-columns, 
generally distinct from zero.
The distribution $W$ somehow depends on the $p$-columns
$a_1,a_2,\dots,a_m$. Let us show that the distribution $W$ depends
on the noted $p$-columns through a so-called parameter of
noncentrality: the \m pp-matrix
$$
\Delta=\sum_{i=1}^m a_i a_i^T.
$$
Let us introduce the \m pm-matrices
\begin{align*}
{\xi}&=\|\xi_1,\xi_2,\dots,\xi_m\|,\\
\es{A}&=\|a_1,a_2,\dots,a_m\|.
\end{align*}
In these notations
$$
W=\p{\mathbf{\xi}+\es A}{\mathbf{\xi}+\es A}.
$$

Let $C$ be an arbitrary orthogonal \m mm-matrix.
Say
$\eta =\xi C$.
Note that $\eta\stackrel{\mathrm{d}}{=}\xi$, and
$$
W\stackrel{\mathrm{d}}{=}\p{\eta+A C}{\eta+ A C}.
$$
We see that the noncentral Wishart distribution  depends on
$\es{A}=\|a_1,\dots,a_m\|$ not directly but through the maximal
invariant $\es{A}$ under orthogonal transformations, that is through
$\p {\es A}{\es A}=\sum_{i=1}^m a_i a_i^T$.

Therefore, in the general case
$$
\p{\prj{\ML_i}{X}}{\prj{\mathcal
L_i}{X}}=W_p(m_i,\Sigma,\Delta_i),
$$
where $\Delta_i=\p{\prj{\ML_i}{\e X}}{\prj{\ML_i}{\e X}}$.
\end{proof}

Let us return to the unbiased estimate of parameter $\Sigma$ of linear
models.
In  linear model \eqref{p2-31} $\pr{L^\perp}{\e X}=0$.
Therefore the statistic \eqref{p2-34} is
$$
\frac{1}{\dim \ML^\bot}\, \p{\prj{\ML^\bot}{X}}{\prj{\mathcal
L^\bot}{X}}=
\frac{1}{n-m} \Sigma^{\frac12} W_p(n-m,I) \Sigma^{\frac12}.
$$
It is obvious that its expected value is $\Sigma$.

\subsection{Testing Linear Hypotheses}
\label{p25}

Copying the univariate linear model, we shall define the
hypothesis in the multivariate linear model \eqref{p2-31} as
\begin{equation}\label{w1}
H: \e\x\in\ML_1,
\end{equation}
where $\ML_1$ is a given submodule such that
$\ML_1\subset \ML$.

In this section we will propose statistics which may serve as the
base for the construction of statistical criteria for testing $H$
\eqref{w1}, free (under $H$) from the parameters $\es M$,
$\Sigma$.

Let us introduce the submodule $\ML_2$ which is an orthogonal
complement $\ML_1$ with respect to $\ML$:
\begin{equation}\label{w2}
\ML=\ML_1\oplus\ML_2.
\end{equation}
Let us consider the decomposition of the space \s pn into three
pairwise orthogonal subspaces:
$$
\s pn=\ML_1\oplus\ML_2\oplus\ML^\perp.
$$
Following theorem~\ref{t5} the random matrices
$$
S_1:=\p{\prj{\ML^\perp}{X}}{\prj{\ML^\perp}{X}}
\qquad \mbox{and}\qquad S_2:=\p{\prj{\ML_2}{X}}{\prj{\mathcal
L_2}{X}}
$$
are  independent and have Wishart distributions.
Regardless of $H$
\begin{equation}\label{w3}
S_1=\p{\prj{\ML^\perp}{X}}{\prj{\mathcal
L^\perp}{X}}=W_p(n-m,\Sigma).
\end{equation}

If the hypothesis $H$ \eqref{w1} is true, then
\begin{equation}\label{w4}
S_2=\p{\prj{\ML_2}{X}}{\prj{\mathcal
L_2}{X}}=W_p(m_2,\Sigma).
\end{equation}
(Here and further we denote $m=\dim\ML$,
$m_1=\dim\ML_1$, $m_2=\dim\ML_2$).

Under the alternative to $H$ \eqref{w1}, the Wishart distribution of statistic
\eqref{w4} becomes noncentral with the parameter of noncentrality
$$
\Delta=\p{\prj{\ML_2}{\e\x}}{\prj{\ML_2}{\e\x}}.
$$
The noncentrality parameter shows
the degree of violation of the hypothesis $H$
\eqref{w1}: $\e\x\in\ML_1$.

In the one-dimensional case (when $p=1$) the statistics \eqref{w3} and
\eqref{w4} turn into random variables distributed as
$\sigma^2 \chi^2(n-m)$ and $\sigma^2 \chi^2(m_2)$ respectively. Their ratio (under
the hypothesis) is distributed free, and therefore it can be used
as a statistical criterion for testing $H$.
This is the well-known F-ratio of Fischer.

In the multivariate case the analogue of F-ratio should be the
``ratio'' of \m pp-matrices $S_2$ and $S_1$.
Under $n-m\ge p$ the matrix $S_1$ \eqref{w3} is non-degenerate, and
therefore there exists a statistic ($\m pp$-matrix)
\begin{equation}\label{w5}
\p{\prj{\ML_2}{X}}{\prj{\ML_2}{X}}\ \p{\prj{\mathcal
L^\perp}{X}}{\prj{\ML^\perp}{X}} ^{-1}
\end{equation}

Unlike the one-dimensional case ($p=1$) the statistic \eqref{w5} is
not distributed free. By distribution, \eqref{w5} is equal to
\begin{equation}\label{w6}
\Sigma^{\frac12}\ W_p(m_2,I)\ W_p^{-1}(n-m,I)\ \Sigma^{-\frac12}.
\end{equation}

However the eigenvalues of matrix \eqref{w5} under the hypothesis $H$
\eqref{w1} are distributed free (from $\es M$, $\Sigma$).
These eigenvalues coincide with the roots of the equation relative to
$\lambda$
\begin{equation}\label{w7}
\det (W_p(m_2,I)-\lambda W_p(n-m,I))=0.
\end{equation}
Therefore certain functions of the roots of equation \eqref{w7}
are traditionally used as critical statistics in testing linear
hypotheses.

Here our investigation enters the traditional realm of multivariate
statistical analysis, and therefore must be finished.

I thank E. Sukhanova, A. Sarantseva, and P. Panov for discussions and
assistance.

The research is supported by RFBR, project 06-01-00454.

\end{document}